\documentclass[12pt]{amsart}
\usepackage[]{hyperref}
\usepackage{amssymb,amsfonts,amsmath,amsopn,amstext,amscd,latexsym, amsthm, enumerate,mathrsfs}
\usepackage{lipsum}
\usepackage[margin=32mm]{geometry}
\usepackage{bbm}
\usepackage[all,cmtip]{xy}

\usepackage{enumerate}
\usepackage[shortlabels]{enumitem}

\newtheorem{theorem}{Theorem}[section]
\newtheorem{thm}[theorem]{Theorem}

\newtheorem*{thmA}{Theorem~A}
\newtheorem*{thmB}{Theorem~B}
\newtheorem*{thmC}{Theorem C}

\newtheorem{lem}[theorem]{Lemma}

\newtheorem{prop}[theorem]{Proposition}

\theoremstyle{remark}

\newtheorem*{ex}{\textrm{Example}}

\DeclareMathOperator{\Irr}{Irr}

\DeclareMathOperator{\SL}{SL}

\newcommand{\Syl}{{\mathrm {Syl}}}

\DeclareMathOperator{\Sym}{Sym}

\DeclareMathOperator{\PSL}{PSL}

\newcommand{\OO}{\mathbf{O}}
\newcommand{\Centralizer}{\mathbf{C}}
\newcommand{\Center}{\mathbf{Z}}

\numberwithin{equation}{section}

\newcommand{\Alt}{{\mathrm {Alt}}}

\DeclareMathOperator{\Real}{Re}

\begin{document}

\title[]{Real class sizes}

\author[H. P. Tong-Viet]{Hung P. Tong-Viet}
\address{Department of Mathematical Sciences, Binghamton University, Binghamton, NY 13902-6000, USA}
\email{tongviet@math.binghamton.edu}



\date{\today}


\begin{abstract} 
In this paper, we study the structures of finite groups using some arithmetic conditions on the sizes of real conjugacy classes.  We prove that a finite group is solvable if  the prime graph on the real class sizes of the group is disconnected.   Moreover, we show that if the sizes of all non-central real conjugacy classes of a  finite group $G$ have the same $2$-part and the Sylow $2$-subgroup of $G$  satisfies certain condition, then $G$ is solvable.  
 \end{abstract}

\maketitle

\section{Introduction}

Let $G$ be a finite group. An element $x\in  G$ is said to be real if there exists an element $g\in G$ such that $x^g=x^{-1}$.  We denote by $\Real(G)$ the set of all real elements of $G$. A conjugacy class $x^G$ containing $x\in G$ is said to be real  if $x$ is a real element of $G$ or equivalently $x^G=(x^{-1})^G$. The size of a real conjugacy class is called a real class size. Several arithmetic properties of the real class sizes can be conveniently stated using  graph theoretic language.
The prime graph on the real class sizes of a finite group $G$, denoted by $\Delta^*(G)$, is a simple graph with vertex set $\rho^*(G)$ the set of primes dividing the size of some real conjugacy class of $G$ and there is an edge between two vertices $p$ and $q$ if and only if the product $pq$ divides some real class size.

Now a prime $p$ is not a vertex of $\Delta^*(G)$, that is, $p\not\in\rho^*(G)$ if and only if $p$ divides no real class size of $G$. In \cite{DNT}, the authors show that $2$ is not a vertex of $\Delta^*(G)$ if and only if $G$ has a normal Sylow $2$-subgroup $S$ (i.e., $G$ is $2$-closed) and $\Real(S)\subseteq\Center(S)$. For odd primes, a similar result is not that satisfactory. Combining results in \cite{GNT}, \cite{IN} and \cite{Tiep}, we can show that if an  odd prime $p$ is not a vertex of $\Delta^*(G)$ and assume further that when $p=3$, $\SL_3(2)$ is not a composition factor of $G$, then $\OO^{2'}(G)$ has a normal Sylow $p$-subgroup and $\OO^{p'}(G)$ is solvable, in particular, $G$ is $p$-solvable (see Lemma \ref{lem: Ito-Michler theorem for conjugacy classes}).
The proofs of the aforementioned results, especially, for odd primes, are quite involved and depend heavily on the classification of finite simple groups. This is in contrast to the similar result for all conjugacy classes, that is, if a prime $p$ does not divide the size of any conjugacy classes of $G$ then $G$ has a normal central Sylow $p$-subgroup. The proof of this classical result is just an application of Jordan's theorem on the existence of derangements in finite permutation groups. 

It is proved in \cite{DNT} that $\Delta^*(G)$ has at most two connected components. In our first result, we will show that if $\Delta^*(G)$ is disconnected, then $G$ is solvable.

\begin{thmA} Let $G$ be a finite group. If the prime graph on the real class sizes of $G$ is disconnected, then $G$ is solvable.
\end{thmA}

We next study in more detail the real class sizes of finite groups with disconnected prime graph on real class sizes.

\begin{thmB}
Let $G$ be a finite group. Suppose that $\Delta^*(G)$ is disconnected. Then $2$ divides some real class size and one of the following holds.
\begin{enumerate}
\item[$(1)$] $G$ has a normal Sylow $2$-subgroup. 

\item[$(2)$] $\Delta^*(\OO^{2'}(G))$ is disconnected and the real class sizes of $\OO^{2'}(G)$ are either odd or powers of $2$.
\end{enumerate}
\end{thmB}

It follows from Theorem B that if the prime graph $\Delta^*(G)$ of a finite group $G$ is disconnected, then $2$ must be a vertex of $\Delta^*(G)$. This confirms once again  the importance of the prime $2$ in the study of real conjugacy classes of finite groups.
In the second conclusion of Theorem B, both connected components of $\Delta^*(\OO^{2'}(G))$ are complete and one of the components  of this graph contains the prime $2$ only. (See Theorem \ref{th: 2 powers or odd}). 
  We should mention that it was proved in \cite{DPS} that a finite group whose all real class sizes are either odd or powers of $2$  is solvable.  However, in the proof of (2), we will need the solvability from Theorem A. So, if one can prove  part (2) of Theorem B without using  the solvability of the group, then we would have another proof of Theorem A.

In \cite{DPS}, it is shown that if all non-central real  conjugacy classes of a group have prime sizes, then the group has a normal Sylow $2$-subgroup or a normal $2$-complement. The next example shows that this is not the case if we only assume that all real class sizes are prime powers. 

\begin{ex} Let  $G=\Alt_4:C_4$ be a solvable group of order $48$. We have $G=\OO^{2'}(G)$, $G/\Center(G)\cong\Sym_4,G/\OO_2(G)\cong\Sym_3$ and the real class sizes of $G$ are $1,3$ or $8$. Clearly, $G$ has no normal Sylow $2$-subgroup nor normal $2$-complement. 
\end{ex}

It follows from \cite{BHM} that if the prime graph defined on all class sizes of a finite group $G$ is disconnected, then $G/\Center(G)$ is a Frobenius group with abelian kernel and complement. By our example above, this does not hold for the prime graph on real class sizes.

In our last result, we provide further evidence for a conjecture proposed in \cite{Tong}. We will prove Conjecture C in \cite{Tong} under some condition on the Sylow $2$-subgroups.

\begin{thmC}
Let $G$ be a finite group. Suppose that the sizes of all non-central real conjugacy classes of $G$ have the same $2$-part. Assume further that $G$ has a Sylow $2$-subgroup $S$ with $\Real(S)\subseteq \Center(S)$. Then $G$ is solvable and $\OO^{2'}(G)$ has a normal $2$-complement. 
\end{thmC}

A conjecture due to G. Navarro states that a finite group $G$ is solvable if $G$ has at most two real class sizes. Clearly, our Theorem C implies this conjecture with an additional assumption that  $\Real(S)\subseteq\Center(S)$ for some Sylow $2$-subgroup $S$ of $G$.  Finite $2$-groups $S$ with $\Real(S)\subseteq\Center(S)$ have been studied by Chillag and Mann \cite{CM}. These are exactly the  finite $2$-groups $S$ for which if $x,y\in S$ and $x^2=y^2$, then $x\Center(S)=y\Center(S)$.

\section{Real conjugacy classes}
Our notation are more or less standard. If $n$ is a positive integer, then $\pi(n)$ is the set of prime divisors of $n$. If $\pi(n)\subseteq \sigma$ for some set of primes $\sigma$, then $n$ is said to be a $\sigma$-number.  If $n>1$ is an integer and $p$ is a prime, then the $p$-part of $n$, denoted by $n_p$, is the largest power of $p$ dividing $n$.

Recall that $\Real(G)$ is the set of all real elements of $G$. 
We collect some properties of real elements and real class sizes in the following lemma.

\begin{lem}\label{lem: real elements}  Let $G$ be a finite group and let $N\unlhd G$.

\begin{enumerate}
\item[$(1)$] If $x\in \Real(G)$, then every power of $x$ is also real.

\item[$(2)$] If $x\in \Real(G)$, then $x^t=x^{-1}$ for some $2$-element $t\in G$.
\item[$(3)$] If $x\in \Real(G)$ and $|x^G|$ is odd, then $x^2=1.$
\item[$(4)$] If $x,y\in \Real(G)$, $xy=yx$ and $(|x^G|,|y^G|)=1$, then $xy\in \Real(G)$. Furthermore, if $(o(x),o(y))=1$, then $\pi(|x^G|)\cup\pi(|y^G|)\subseteq \pi(|(xy)^G|)$.
\item[$(5)$] If $|G:N|$ is odd, then $\Real(G)=\Real(N)$.

\item[$(6)$] Suppose that $Nx$ is a real element in $G/N$. If $|N|$ or the order of $Nx$ in $G/N$ is odd, then $Nx=Ny$ for some real element $y\in G$ (of odd order if the order of $Nx$ is odd).
\end{enumerate}
\end{lem}

\begin{proof} Let $x\in G$ be a real element. Then $x^g=x^{-1}$ for some $g\in G$. If $k$ is any integer, then $(x^k)^g=(x^g)^k=(x^{-1})^k=(x^{k})^{-1}$ so $x^k$ is real which proves (1). Write $o(g)=2^am$ with $(2,m)=1$ and let $t=g^m$. Then $t$ is a $2$-element and $x^t=x^{g^m}=x^g=x^{-1}$ as $g^2\in\Centralizer_G(x)$  and $m$ is odd. This proves (2).

 Parts (3)--(5) are in Lemma~6.3 of \cite{DNT}.  Finally, (6) is Lemma 2.2 in \cite{GNT}.
\end{proof}

Fix $1\neq x\in \Real(G)$, set $\Centralizer^*_G(x)=\{g\in G\:|\:x^g\in \{x,x^{-1}\}\}.$ Then $\Centralizer^*_G(x)$ is a subgroup of $G$ containing $\Centralizer_G(x)$. If $g\in G$ such that $x^g=x^{-1}$, then $x^{g^2}=x$ so $g^2\in\Centralizer_G(x)$. Assume that $x$ is not an involution. We see that $g\in\Centralizer_G^*(x)\setminus \Centralizer_G(x)$  and if $h\in \Centralizer_G^*(x)\setminus \Centralizer_G(x)$, then $x^h=x^{-1}=x^g$ so that $hg^{-1}\in\Centralizer_G(x)$ or equivalently $h\in\Centralizer_G(x)g$. Thus $\Centralizer_G(x)$ has index $2$ in $\Centralizer_G^*(x)$ and hence $|x^G|$ is even. In particular, this is the case if $x$ is a nontrivial real element of odd order.

The next lemma is well-known. We will use this lemma freely without further reference.

\begin{lem}\label{lem:div} Let $G$ be a finite group and let $N\unlhd G$. Then

\begin{enumerate}
\item[$(1)$] If $x\in N$, then $|x^N|$ divides $|x^G|$.

\item[$(2)$] If $Nx\in G/N$, then $|(Nx)^{G/N}|$ divides $|x^G|$.
\end{enumerate}
\end{lem}

The following lemma shows that a finite group $G$ has no nontrivial real element of odd order if and only if $G$ has a normal Sylow $2$-subgroup.

\begin{lem}\label{lem:even order real} \emph{(\cite[Proposition 6.4]{DNT}).} The following are equivalent:
\begin{enumerate}
\item[$(1)$] Every nontrivial element in $\Real(G)$ has even order.
\item[$(2)$] Every element in $\Real(G)$ is a $2$-element.
\item[$(3)$] $G$ has a normal Sylow $2$-subgroup.
\end{enumerate}
\end{lem}

The next lemma determines the number of connected components of the prime graphs on  real class sizes.
\begin{lem}\label{lem:real  prime graph components}\emph{(\cite[Theorem ~6.2]{DNT}).}
For any finite group $G$, $\Delta^*(G)$ has at most two connected components.
\end{lem}

If a finite group $G$ is of even order, then it has a real element of order $2$. If an odd prime $p$ dividing $|G|$, $G$ may not have a real element of order $p$. However, if $G$ has no proper normal subgroup of odd index and $G$ is $p$-solvable,  Dolfi, Malle and Navarro \cite{DMN} show that $G$ must contain a real element of order $p$.  

\begin{lem}\label{lem:real element of order p}\emph{(\cite[Corollary~B]{DMN}).}
Let $G$ be a finite group with $\OO^{2'}(G)=G$. Suppose that $p$ is an odd prime dividing $|G|$. If $G$ is $p$-solvable, then $G$ has a real element of order $p$.
\end{lem}

In the next two lemmas, we state the  It\^{o}-Michler theorem for real conjugacy classes.
\begin{lem}\label{lem:odd size}\emph{(\cite[Theorem~6.1]{DNT}).}
Let $G$ be a finite group and let $P$ be a Sylow $2$-subgroup of $G$. Then all real classes of $G$ have odd size if and only if $P\unlhd G$ and $\Real(P)\subseteq \Center(P)$.
\end{lem}

\begin{lem}\label{lem: Ito-Michler theorem for conjugacy classes}
Let  $G$ be a finite group and $p$ be an odd prime. If $p=3$, assume in addition that $G$ has no composition factor isomorphic to $\SL_3(2)$. If $p$ does not divide $|x^G|$ for all real elements of $G$, then $G$ is $p$-solvable and $\OO^{p'}(G)$ is solvable. Furthermore, $\OO^{2'}(G)$ has a normal Sylow $p$-subgroup $P$ and $P'\leq \Center(\OO^{2'}(G))$.
\end{lem}

\begin{proof}
The first claim is Theorem A in \cite{GNT}. Now, Theorem B in that reference implies that $p$ does not divide $\chi(1)$ for all real-valued irreducible characters $\chi\in\Irr(G)$. By Theorem A in \cite{Tiep}, we know that $\OO^{p'}(G)$ is solvable. Finally, the last statement follows from Theorem A in \cite{IN}.
\end{proof}

\section{Proofs of Theorems A and B}

Let $G$ be a finite group. Suppose that $\Delta^*(G)$ is disconnected. Then $\Delta^*(G)$ has exactly two connected components by Lemma \ref{lem:real  prime graph components}. The following will be used frequently in our proofs.

\begin{lem}\label{lem:2-closed} Let $G$ be a finite group and suppose that $\Delta^*(G)$ has two connected components with vertex sets $\pi_1$ and $\pi_2$, where $2\not\in\pi_2$.  Then there exists an involution $i\in G$ such that $|i^G|>1$ is a $\pi_2$-number and $\Centralizer_G(i)$ has a normal Sylow $2$-subgroup. 
\end{lem}

\begin{proof} Let $p$ be a prime in $\pi_2$. Then $p$ must divide $|i^G|$ for some nontrivial real element $i\in G$.  Clearly, every prime divisor of $|i^G|$ is adjacent to $p\in\pi_2$, this implies that $|i^G|$ is a nontrivial $\pi_2$-number.
Hence $|i^G|>1$ is odd and so $i^2=1$ by Lemma \ref{lem: real elements}(3) and thus $i$ is an involution of $G$.
 
Assume that $\Centralizer_G(i)$ has a nontrivial real element $x$ of odd order. Then  $xi=ix$ and $|x^G|$ is even so $|x^G|$ is a $\pi_1$-number and thus $(|x^G|,|i^G|)=1$. Lemma \ref{lem: real elements}(4) implies that $xi$ is a real element. Furthermore,  since $(o(x),o(i))=1$,  $2p$ divides $|(ix)^G|$ by Lemma \ref{lem: real elements}(4) again. This means that $2\in \pi_1$ and $p\in\pi_2$ are adjacent in $\Delta^*(G)$, which is impossible. Therefore, $\Centralizer_G(i)$ has no nontrivial real element of odd order and thus it has a normal Sylow $2$-subgroup by Lemma \ref{lem:even order real}. 
\end{proof}

Notice that if $N\unlhd G$, then $\Delta^*(N)$ is a subgraph of $\Delta^*(G)$ by Lemma \ref{lem:div}(1) and the fact that $\Real(N)\subseteq \Real(G)$. However, in general, it is not true that $\Delta^*(G/N)$ is a subgraph of $\Delta^*(G)$. The involutions in $G/N$ might produce extra vertices as well as edges in $\Delta^*(G/N)$. However, this is the case if $|N|$ is odd.

\begin{lem}\label{lem:subgraph} Let $G$ be a finite group and let $N\unlhd G$ with $|N|$ odd. Then $\Delta^*(G/N)$ is a subgraph of $\Delta^*(G)$.
\end{lem}

\begin{proof}  We first show that $\rho^*(G/N)\subseteq\rho^*(G)$. Indeed, let $p\in\rho^*(G/N)$ and let $Nx\in G/N$ be a real element such that $p$ divides $|(Nx)^{G/N}|$. By Lemma \ref{lem: real elements}(6), there exists a real element $y\in G$ such that $Nx=Ny$. Since $|(Nx)^{G/N}|=|(Ny)^{G/N}|$ divides $|y^G|$, $p$ divides $|y^G|$, so $p\in\rho^*(G)$.

With a similar argument, we can show that if $p\neq q\in\rho^*(G/N)$ which are adjacent in $\Delta^*(G/N)$, then $p,q$ are adjacent in $\Delta^*(G)$ by using  Lemma \ref{lem: real elements}(6) again. Thus $\Delta^*(G/N)$ is a subgraph of $\Delta^*(G)$.
\end{proof}

Let $X$ be a subgroup or a quotient of a finite group $G$ and suppose that $\Delta^*(X)$ is a subgraph of $\Delta^*(G)$. Assume that $\Delta^*(G)$ is disconnected having two connected components with vertex sets $\pi_1$ and $\pi_2$, respectively.  To show that $\Delta^*(X)$ is disconnected, it suffices to show that $\rho^*(X)\cap\pi_i\neq\emptyset$ for $i=1,2$ or equivalently $X$ has two real elements $u_i,i=1,2$ which both lift to real elements of $G$ and $\pi(|u_i^X|)\cap\pi_i\neq\emptyset$ for $i=1,2$.

For a finite group $G$ and a prime $p,$ $G$ is said to be $p$-closed if it has a normal Sylow $p$-subgroup and it is $p$-nilpotent if it has a normal $p$-complement.

\begin{prop}\label{prop:normal and quotient} Let $G$ be a finite group.  Suppose that $\Delta^*(G)$ has two connected components with vertex sets $\pi_1$ and $\pi_2$ where $2\not\in\pi_2$.  Then
\begin{enumerate}
\item[$(1)$] If $G$ is not $2$-closed and  $M\unlhd G$ with $|G:M|$  odd, then $\Delta^*(M)$ is disconnected.

\item[$(2)$] If $N\unlhd G$ with $|N|$ odd and assume further that $G=\OO^{2'}(G)$ is not $2$-nilpotent, then $\Delta^*(G/N)$ is disconnected.
\end{enumerate}
\end{prop}

\begin{proof} By Lemma \ref{lem:2-closed}, $G$ has an involution $i$ such that $|i^G|>1$ is a $\pi_2$-number and $\Centralizer_G(i)$ has a normal Sylow $2$-subgroup $S$. 

For (1), let $M\unlhd G$ with $|G:M|$ being odd. Notice that $\rho^*(M)\subseteq\rho^*(G)=\pi_1\cup\pi_2$ and that if $M$ is $2$-closed, then $G$ is also $2$-closed. Thus we assume that $M$ is not $2$-closed. By Lemma \ref{lem:odd size}, $2$ divides some real class size of $M$ and hence $2\in\rho^*(M)\cap\pi_1$.  Now $M$ contains every real element of $G$ by Lemma \ref{lem: real elements}(5), in particular, $i\in M$. Moreover, $\Delta^*(M)$ is a subgraph of $\Delta^*(G)$. If $M\leq \Centralizer_G(i)$, then $S\unlhd M$ as $|G:M|$ is odd, so $M$ is $2$-closed, a contradiction. Thus,  $|i^M|>1$ and hence $\rho^*(M)\cap\pi_2\neq\emptyset$.  Therefore, $\Delta^*(M)$ is disconnected as $\rho^*(M)\cap \pi_i$ is non-empty for each $i=1,2$.

 For (2), suppose that $N\unlhd G$ with $|N|$ odd. By Lemma \ref{lem:subgraph}, $\Delta^*(G/N)$ is a subgraph of $\Delta^*(G)$. In particular, $\rho^*(G/N)\subseteq \rho^*(G)=\pi_1\cup \pi_2$. If $G/N$ is $2$-closed, then $SN\unlhd G$ is of odd index and thus $G=SN$ since  $G=\OO^{2'}(G)$. However, this would imply that $G$ is $2$-nilpotent with a normal $2$-complement $N$. Therefore, we assume that $G/N$ is not $2$-closed.

 Clearly, $Ni\in G/N$ is an involution. If $Ni$ is central in $G/N$, then $G/N=\Centralizer_{G/N}(Ni)=\Centralizer_G(i)N/N,$ where the latter equality follows from \cite[Lemma 7.7]{Isaacs-1}. Hence $G/N$ has a normal Sylow $2$-subgroup $SN/N$, a contradiction. Thus $Ni$ is not central in $G/N$ and  $|(Ni)^{G/N}|>1$, so $\rho^*(G/N)\cap\pi_2\neq\emptyset$. Finally, as $G/N$ is not $2$-closed,  $2\in \rho^*(G/N)\cap\pi_1$ by applying Lemma \ref{lem:odd size} again. Therefore, $\Delta^*(G/N)$ is disconnected.
\end{proof}

We are now ready to prove Theorem A which we restate here.
\begin{thm}\label{th:disconnected real class sizes graph}
Let $G$ be a finite group. If $\Delta^*(G)$ is disconnected, then $G$ is solvable.\end{thm}

\begin{proof}

Let $G$ be a counterexample with minimal order. Then  $G$ is non-solvable and $\Delta^*(G)$ is disconnected. By Lemma \ref{lem:real  prime graph components}, $\Delta^*(G)$ has exactly two connected components with vertex sets $\pi_1$ and $\pi_2$, respectively. If $G$ has a normal Sylow $2$-subgroup, then it is clearly solvable by Feit-Thompson theorem. Therefore, we can assume that $G$ has no normal Sylow $2$-subgroup. Now it follows from Lemma \ref{lem:even order real}  that $G$ has a nontrivial real element $x$ of odd order. Then $|x^G|$ is divisible by $2$ and hence $2$ is always a vertex of $\Delta^*(G)$.  We assume that $2\in \pi_1$.  Hence all vertices in $\pi_2$ are odd primes.

\medskip
(1) By Lemma \ref{lem:2-closed}, $G$ has an involution $i$ such that $|i^G|>1$ is a $\pi_2$-number and $\Centralizer_G(i)$ has a normal Sylow $2$-subgroup, say $S$. Clearly, $S$ is also a Sylow $2$-subgroup of $G$ as $|i^G|$ is odd. Notice that $\Centralizer_G(i)$ is solvable. Now $G$ has a nontrivial real element $x$ of odd order by Lemma \ref{lem:even order real}. Clearly $|x^G|$ is even so  $|x^G|$ must be a $\pi_1$-number. Thus $(|x^G|,|i^G|)=1$; therefore, $G$ is not a nonabelian simple group by  \cite[Theorem 2]{FA}.

\medskip
 (2) $G=\OO^{2'}(G)$.
   Since $G$ is not $2$-closed, $\Delta^*(\OO^{2'}(G))$ is disconnected by Proposition \ref{prop:normal and quotient}(1). If $\OO^{2'}(G)<G$, then $\OO^{2'}(G)$ is solvable by the minimality of $|G|$, hence $G$ is solvable since $G/\OO^{2'}(G)$ is solvable. This contradiction shows that $G=\OO^{2'}(G)$.

 \medskip
 (3) $\OO_{2'}(G)=1$.  By (2) above, we have $G=\OO^{2'}(G)$ and since $G$ is not solvable, $G$ is not $2$-nilpotent so that by Proposition \ref{prop:normal and quotient}(2) $\Delta^*(\overline{G})$ is disconnected, where $\overline{G}=G/\OO_{2'}(G)$. If $\OO_{2'}(G)$ is nontrivial, then $|\overline{G}|<|G|$ and thus by the minimality of $|G|$, $\overline{G}$ is solvable and so is $G$. Hence $\OO_{2'}(G)=1$ as required.

\smallskip
(4)  If $M$ is a maximal normal subgroup of $G$, then  $|G:M|=2$ and $G=M\langle i\rangle$. Let $M$ be a maximal normal subgroup of $G$. Then $G/M$ is a simple group.

(a) Assume first that $G/M$ is abelian, then $G/M\cong C_r$ for some prime $r$. Since $\OO^{2'}(G)=G$ by (2), we deduce that $r=2$. Clearly, $M$ is not solvable. We next claim that $i\not\in M$ and hence we have that $G=M\langle i\rangle$.

Suppose that by contradiction that $i\in M$. Then either $|i^M|>1$ or $M=\Centralizer_G(i)$. If the latter case holds, then $M$ is solvable by (1) and thus $G$ is solvable, a contradiction. Assume that $|i^M|>1$. Notice that $\Delta^*(M)$ is a subgraph of $\Delta^*(G)$. We see that $\rho^*(M)\cap \pi_2\neq\emptyset$ as every prime divisor of $|i^M|$ is in $\pi_2$ and $i\in\Real(M)$. Observe next that $M$ is not $2$-closed so it has a nontrivial real element $z$ of odd order by Lemma \ref{lem:even order real} and thus $|z^M|$ is even. In other words, $2\in\rho^*(M)$. Therefore $\rho^*(M)\cap\pi_1\neq\emptyset$. Hence we have shown that $\Delta^*(M)$ is disconnected. Thus by induction, $M$ is solvable, which is a contradiction.

\smallskip
(b) Now,  assume that $G/M$ is a non-abelian simple group.  Set $\overline{G}=G/M$.

Assume first that $i\not\in M$. Then $\overline{i}$ is an involution in $\overline{G}$ and $|\overline{i}^{\overline{G}}|$ divides $|i^G|$. Let $\overline{y}$ be a nontrivial real element of $\overline{G}$ of odd order (such an element exists by Lemma \ref{lem:even order real}). By Lemma \ref{lem: real elements}(6), $\overline{y}$ lifts to a real element $z\in G$ of odd order. Therefore $|\overline{y}^{\overline{G}}|$ divides $|z^G|$. Since $(|z^G|,|i^G|)=1$, we deduce that $(|\overline{y}^{\overline{G}}|,|\overline{i}^{\overline{G}}|)=1$, contradicting \cite[Theorem 2]{FA}.

Assume that $i\in M$. If $G=M\Centralizer_G(i)$, then $G/M\cong \Centralizer_G(i)/(M\cap\Centralizer_G(i))$ is a non-abelian simple group, which is impossible as $\Centralizer_G(i)$ is solvable by (1). Thus $H:=M\Centralizer_G(i)<G$ and $|G:H|=|\overline{G}:\overline{H}|$ divides $|G:\Centralizer_G(i)|=|i^G|$. Let $\overline{y}\in \overline{G}$ be a real element of odd order and $z\in G$ be a real element of odd order such that $\overline{y}=\overline{z}$. Then $|\overline{y}^{\overline{G}}|=|\overline{z}^{\overline{G}}|$ divides $|z^G| $ so $(|\overline{G}:\Centralizer_{\overline{G}}(\overline{z})|,|\overline{G}:\overline{H}|)=1$ as $(|z^G|,|i^G|)=1$. Therefore, $\overline{G}=\overline{H}\Centralizer_{\overline{G}}(\overline{z})$, where $\overline{H}$ has odd index in $\overline{G}$ and $\overline{H}$ has a normal Sylow $2$-subgroup $\overline{S}$.

 As every nontrivial real element of odd order of $\overline{G}$ lifts to a nontrivial real element of odd order of $G$ by Lemma \ref{lem: real elements}(6), every prime divisor $s$ of $|\overline{G}:\overline{H}|$ (which lies in $\pi_2$) divides the size of no nontrivial real element of odd order of $\overline{G}$, so $\overline{G}\cong \PSL_2(q)$ with $q=2^r-1$ a Mersenne prime and $s\mid (q-1)/2$ or $\overline{G}\cong\textrm{M}_{23}$ with $s=5$ by \cite[Theorem 4.1]{GNT}. If the latter case holds, then $|\overline{G}:\overline{H}|$ must be a power of $5$. Inspecting \cite{Atlas} shows that this is not the case. Thus $\overline{G}\cong \PSL_2(q)$ with $q=2^r-1$ a Mersenne prime. It follows that $r\ge 3$ and $q\ge 7$ as $\overline{G}$ is non-solvable. Observe that $\overline{S}$ is a Sylow $2$-subgroup of $\overline{G}$ of order $q+1=2^r\ge 8$ and so $\overline{S}\cong \textrm{D}_{q+1}$, which is a maximal subgroup of $\overline{G}$ unless $q=7$. Suppose first that $q>7$. It implies that $\overline{H}=\overline{S}$ so that $|\overline{G}:\overline{H}|=q(q-1)/2$.  In particular, $q$ divides $|\overline{G}:\overline{H}|$ and thus $q\mid (q-1)/2$ by the claim above, which is impossible. 
 Now assume that $q=7$. Notice that the Sylow $2$-subgroup of $\overline{G}\cong\PSL_2(7)$ is self-normalizing but not maximal. Since $\overline{S}\unlhd \overline{H}$, we must have that $\overline{H}=\overline{S}$ and  we will get a contradiction as above.  
 
 \smallskip
(5) $G=G'\langle i\rangle$ and $|G:G'|=2$.

Since $G$ is not non-abelian simple, $G$ possesses a maximal normal subgroup $W$. It follows from (4) that $G=W\langle i\rangle$ and $|G:W|=2$. It also follows from (4) that $i$ does not lie in any maximal normal subgroup of $G$ so that $G=\langle i^G\rangle.$

Clearly $G'\leq W$ as $G/W\cong C_2$ is abelian. Moreover, $G=\langle i^G\rangle \leq G'\langle i\rangle\leq W\langle i\rangle=G.$
It follows that $W=G'$ and the claim follows.

\medskip   
(6)  $\OO^{2'}(G')$ is a $\pi_1$-group.

Since $|G:G'|=2$ and $G$ is non-solvable, $G'$ is non-solvable and thus $\Delta^*(G')$ is a connected subgraph of $\Delta^*(G)$ with  $2\in\rho^*(G')\subseteq \pi_1$. Observe that $\pi(G)=\pi(G')$. Let $\sigma=\pi(G')\setminus\pi_1$. Now,  if $q\in \sigma$, then $q$ is odd and divides the size of no nontrivial real conjugacy classes of $G'$. 
Let $K=\OO^{2'}(G')$. 

If $q\in \sigma$ and $q>3$ or $q=3$ and $\SL_3(2)$ is not a composition factor of $G'$, then $K$ has a normal Sylow $q$-subgroup $Q$ by Lemma \ref{lem: Ito-Michler theorem for conjugacy classes}.  Clearly $Q\unlhd G$ and thus $Q=1$ by (3). Hence $q$ does not divide $ |K|$. 

We now suppose that $q=3$ and $\SL_3(2)$ is isomorphic to a composition factor of $G'$. As $G'/K$ is solvable, $\SL_3(2)$ is isomorphic to a composition factor of $K$. Let $L$ be a subnormal subgroup of $K$ and $U\unlhd L$ such that $L/U\cong\SL_3(2)$. Using \cite{Atlas}, we see that $L/U$ has a self-normalizing Sylow $2$-subgroup $T/U$ and a real element $Uz\in L/U$  of order $3$ with $|(Uz)^{L/U}|=7\cdot 8$. There exists a real element $y\in L$ of $3$-power order  with $Uz=Uy$ (see \cite[Lemma 2.6]{Tong}). Since $|(Uz)^{L/U}|$ divides $|y^L|$ and $|y^L|$ divides $|y^G|$, we see that  $7\in \pi_1$.

By (1), $\Centralizer_G(i)$ has a normal Sylow $2$-subgroup $S$ with $S\in\Syl_2(G)$.  As $L$ is subnormal in $G$, $S\cap L$ is a Sylow $2$-subgroup of $L$ and thus $(S\cap L)U/U$ is a Sylow $2$-subgroup of $L/U$. Observe that $\Centralizer_{G}(i)$
contains a Sylow $7$-subgroup, say $Q$, of $G$. Hence $(Q\cap L)U/U$ is a Sylow $7$-subgroup of $L/U$. Since $Q$ normalizes $S$, we can see that $(Q\cap L)U/U$ normalizes $(S\cap L)U/U$, which is impossible as the Sylow $2$-subgroup of $L/U$ is self-normalizing.
Thus we have shown that $K=\OO^{2'}(G')$ is a $\pi_1$-group.

\medskip
 \textbf{The final contradiction}.  By (5), we have $G=G'\langle i\rangle$ so $\Centralizer_G(i)=\Centralizer_{G'}(i)\langle i\rangle$ which implies that $|i^G|=|G:\Centralizer_G(i)|=|G':\Centralizer_{G'}(i)|$.
Since $|i^G|$ is a $\pi_2$-number and $\Centralizer_{G}(i)$ is solvable, $\Centralizer_{G'}(i)$ is also solvable and possesses a Hall $\pi_1$-subgroup $T$ which is also a Hall $\pi_1$-subgroup of $G'$ (as $|G':\Centralizer_{G'}(i)|$ is a $\pi_2$-number). As $\OO^{2'}(G')\unlhd G'$ is a $\pi_1$-subgroup by (6), we deduce that $\OO^{2'}(G')\leq T$ and thus $\OO^{2'}(G')$ is solvable since $T\leq \Centralizer_{G'}(i)$ is solvable. Clearly, $G'/\OO^{2'}(G')$ is solvable by Feit-Thompson theorem, which implies that $G'$ is solvable and hence $G$ is solvable. This contradiction finally proves the theorem.
\end{proof}

In the next result, we show that if $G=\OO^{2'}(G)$ and $\Delta^*(G)$ is disconnected, then each connected component of $\Delta^*(G)$ is complete and one of the components is just $\{2\}$. In particular, every real class size of $G$ is either odd or a $2$-power.

\begin{thm}\label{th: 2 powers or odd}
Let $G$ be a finite group. Suppose that $G=\OO^{2'}(G)$ and $\Delta^*(G)$ is disconnected with vertex sets $\pi_1$ and $\pi_2$ where $2\not\in\pi_2$. Then
 $\pi_1=\{2\}$ and $\pi_2=\pi(|i^G|)$ for some  non-central involution $i\in G$.
\end{thm}

\begin{proof} By Theorem \ref{th:disconnected real class sizes graph}, we know that $G$ is solvable. Let $i\in G$ be an involution as in Lemma \ref{lem:2-closed} and let $S$ be a normal Sylow $2$-subgroup of $\Centralizer_G(i)$.
Let $\sigma$ be the set of odd prime divisors $p$ of $|\Centralizer_G(i)|$ such that $p$ does not divide $ |i^G|$. Since $|i^G|$ is a $\pi_2$-number and $|G|=|i^G|\cdot |\Centralizer_G(i)|$, we see that $$\pi_1\setminus\{2\}\subseteq \pi(G)\setminus (\pi_2\cup\{2\})\subseteq \sigma=\pi(G)\setminus (\{2\}\cup \pi(|i^G|)).$$

Assume that $\sigma=\emptyset$. Then  $\pi_1\subseteq \{2\}$ and $\pi_2\subseteq\pi(|i^G|)$. As $G=\OO^{2'}(G)$ and $\Delta^*(G)$ is disconnected,  $2$ divides some real class size of $G$, hence $2\in\pi_1$. Moreover $\pi(|i^G|)\subseteq \pi_2$; therefore $\pi_1=\{2\}$ and $\pi(|i^G|)=\pi_2$ as wanted.

Assume next that $\sigma$ is not empty and let $p\in\sigma$. Then $p$ is odd and  $p$ divides $|\Centralizer_G(i)|$ but does not divide $ |i^G|$. By Lemma \ref{lem:real element of order p}, $G$ has a real element $x$ of order $p$. Let $P$ be a Sylow $p$-subgroup of $\Centralizer_G(i)$. Since $|G:\Centralizer_G(i)|=|i^G|$ is not divisible by $p$, $P$ is also a Sylow $p$-subgroup of $G$. Replacing $x$ by its $G$-conjugates, we can assume $x\in P\leq \Centralizer_G(i)$ by Sylow theorem. We have that $xi=ix$, $(o(x),o(i))=1$ and $(|x^G|,|i^G|)=1$ so that by applying Lemma \ref{lem: real elements}(4), $xi$ is real in $G$ and $$\pi(|x^G|)\cup\pi(|i^G|)\subseteq \pi(|(xi)^G|).$$
As $|i^G|>1$, we can find a prime $r$ dividing $|i^G|$ and so $r\in\pi_2$. Now the previous inclusion would imply that $2\in\pi(|x^G|)\subseteq \pi_1$ and $r\in\pi_2$ are adjacent in $\Delta^*(G)$, which is impossible.
\end{proof}

We now consider the situation when $2$ is not a vertex of $\Delta^*(G)$, where $G$ is a finite group. By Lemma \ref{lem:odd size}, $G$ has a normal Sylow $2$-subgroup $S$ with $\Real(S)\subseteq\Center(S)$.

\begin{lem}\label{lem:2 is a vertex} Let $G$ be a finite group. Suppose that $G$ has a normal Sylow $2$-subgroup $S$ with $\Real(S)\subseteq \Center(S)$.  Then $\Delta^*(G)$ is connected.
\end{lem}

\begin{proof} If $\Delta^*(G)$ has at most one vertex, then we are done. So, assume that $\Delta^*(G)$ has at least two vertices, i.e., $|\rho^*(G)|\ge 2$.

We argue by contradiction. Suppose that $\Delta^*(G)$ is not connected. Then $\Delta^*(G)$ has two connected components with vertex sets $\pi_1$ and $\pi_2$. Then we can find two  non-central real elements $x,y\in\Real(G)$ such that  $\pi(|x^G|)\subseteq \pi_1$ and $\pi(|y^G|)\subseteq\pi_2$ so $(|x^G|,|y^G|)=1$.

 Since $S\unlhd G$ and $\Real(S)\subseteq \Center(S)$, $\Real(G)\subseteq \Real(S)\subseteq \Center(S)$ by using Lemma \ref{lem: real elements}(5). Thus all nontrivial real elements of $G$ are involutions in $\Center(S)$.  
 Let $E$ be the set of all involutions of  $\Center(S)$ together with the identity. Then $E$ is an elementary abelian subgroup of $\Center(S)$. Indeed, $E=\Omega_1(\Center(S))$ and thus $E\unlhd G.$ 
 
 Clearly $x,y\in E$ so that $E$ is not central. In particular, $\Centralizer_G(E)\unlhd G$ is a proper subgroup of $G$. Notice that $E\leq \Center(S)$ and hence $S\leq \Centralizer_G(E)$. Therefore $A:=G/\Centralizer_G(E)$ is a nontrivial group of odd order and we can consider $E$ as an $\mathbb{F}_2A$-module.
 By Maschke's theorem, $E$ is a completely reducible $A$-module. Observe that $|x^A|=|A:\Centralizer_A(x)|=|G:\Centralizer_G(x)|=|x^G|$ and $|y^A|=|y^G|$. Therefore, the two $A$-orbits $x^A$ and $y^A$ have coprime sizes. By Theorem 1.1 in \cite{DGPS}, the $A$-orbit $(xy)^A$ has size $|x^A|\cdot |y^A|$. Hence $|(xy)^G|=|x^G|\cdot |y^G|$, where $xy\in E$ is an involution. This implies that there is an edge between a prime in $\pi_1$ and a prime in $\pi_2$, which is impossible.
\end{proof}

We are now ready to prove Theorem B.

\begin{thm}\label{th:structure}
Let $G$ be a finite group. Suppose that $\Delta^*(G)$ is disconnected. Then $2$ divides some real class size and one of the following holds.

\begin{enumerate}
\item[$(1)$] $G$ has a normal Sylow $2$-subgroup. 

\item[$(2)$] $\Delta^*(\OO^{2'}(G))$ is disconnected and the real class sizes of $\OO^{2'}(G)$ are either odd or powers of $2$.
\end{enumerate}
\end{thm}

\begin{proof}
Suppose that $\Delta^*(G)$ is disconnected and let the vertex sets of the connected components are $\pi_1$ and $\pi_2$, respectively. Assume that $2\not\in\pi_2$. We see that $\Delta^*(G)$ has at least two vertices.

We first claim that $2\in\pi_1$. It suffices to show that $2$ divides some real class size of $G$. Suppose by contradiction that $2$ divides no real class size. Then $G$ has a normal Sylow $2$-subgroup $S$ with $\Real(S)\subseteq \Center(S)$ by Lemma \ref{lem:odd size}. However, Lemma \ref{lem:2 is a vertex} implies that $\Delta^*(G)$ is connected, which is a contradiction.

Next, suppose that $G$ has no normal Sylow $2$-subgroup. We claim that part (2) of the conclusion holds. Let $K:=\OO^{2'}(G)$. By Proposition \ref{prop:normal and quotient}(1), $\Delta^*(K)$ is disconnected. Since $\OO^{2'}(K)=K$ and $\Delta^*(K)$ is disconnected, the result follows from Theorem \ref{th: 2 powers or odd}.
\end{proof}

 We suspect that if a finite group $G$ has a normal Sylow $2$-subgroup $S$, then $\Delta^*(G)$ is connected, that is, case (1) in Theorem \ref{th:structure} cannot occur. However, we are unable to prove or disprove this yet. In view of Lemma \ref{lem:2 is a vertex}, this  is true if $\Real(S)\subseteq \Center(S)$.

There are many examples of finite groups whose prime graphs on real class sizes are disconnected. For the first example, let $m>1$ be an odd integer; the dihedral group $\textrm{D}_{2n}$ of order $2n$, where $n=m$ or $n=2m$, has a disconnected prime graph on real class sizes as its real class sizes are just $1,2$ and $m$.  For another example, let $G$ be a Frobenius group with Frobenius kernel $F$ and complement $H$, where both $F$ and $H$ are abelian and $|H|$ is even. In this case, all the nontrivial real class sizes of $G$ are $|F|$ and $|H|$ and since $(|F|,|H|)=1$, $\Delta^*(G)$ is disconnected.

\section{Proof of Theorem C}

Let $G$ be a finite group.  Observe that if $x\in \Center(G)$ is a real element of $G$, then $x^2=1$.  
We first begin with the following lemma.

\begin{lem}\label{lem:real 2-elements}
Let $G$ be a finite group and let $S\in\Syl_2(G)$. Suppose that $\Real(S)\subseteq\Center(S)$ and $|x^G|_2=2^a\ge 2$ for all non-central real elements $x\in G$. If $y$ is a nontrivial real element of $G$ whose order is a $2$-power, then $y$ is a central involution of $G$.
\end{lem}

\begin{proof} Let $y$ be a nontrivial real element of $G$ whose order is a $2$-power. Then $y^t=y^{-1}$ for some $2$-element $t\in G$ by Lemma \ref{lem: real elements}(2). As $t$ normalizes $\langle y\rangle$, $U:=\langle y,t\rangle$ is a $2$-subgroup of $G$. By Sylow theorem, $U^g \leq S$ for some $g\in G.$ If $y^g$ is a central involution of $G$, then so is $y$. Thus we can assume that $U\leq S$. Since $y\in\Real(U)$, we have $y\in \Real(S)\subseteq \Center(S)$ and so $y^2=1$. Hence $y$ is an involution. 
Finally, since $y\in\Center(S)$, $|y^G|$ is odd which forces $y\in\Center(G)$ as by assumption $|x^G|$ is even for all non-central real elements  $x$ of $G$.
\end{proof}

The following lemma is obvious.
\begin{lem}\label{lem:reduction}
Let $G$ be a finite group and let $S\in\Syl_2(G)$. Suppose that $|x^G|_2=2^a\ge 2$ for all non-central real elements $x\in G$. Let $K\unlhd G$ be a normal subgroup of odd index. Then $|x^K|_2=2^a$ for all non-central real elements $x\in K$. 
\end{lem}

\begin{proof}
Let $K$ be a normal subgroup of $G$ of odd index. By Lemma \ref{lem: real elements}(5), we have $\Real(G)\subseteq\Real(K)$. Now let $x\in K$ be a non-central real element of $K$ and let $C:=\Centralizer_G(x)$. Let $P\in\Syl_2(C)$ and let $S\in\Syl_2(G)$ such that $P\leq S$. Since $|G:K|$ is odd, we have $P\leq S\leq K$. In particular, $S\in\Syl_2(K)$. We see that $\Centralizer_K(x)=K\cap C$ and $P\leq K\cap C$. Thus $P\leq \Centralizer_K(x)\leq C$ and so $P$ is also a Sylow $2$-subgroup of $\Centralizer_K(x)$. Therefore $|C|_2=|\Centralizer_K(x)|_2$ and hence $|x^G|_2=|G:C|_2=|S:P|=|K:\Centralizer_K(x)|_2=|x^K|_2$.
\end{proof}

\begin{lem}\label{lem:central involutions}
Let $G$ be a finite group and let $S\in\Syl_2(G)$. Suppose that $|x^G|_2=2^a\ge 2$ for all non-central real elements $x\in G$. Assume further that $G$ has a  Sylow $2$-subgroup $S$ with $\Real(S)\subseteq \Center(S)$. Then every nontrivial real element of $G/\OO_{2'}(G)$ of $2$-power order lies in the center of $G/\OO_{2'}(G)$.
\end{lem}

\begin{proof}
Let $N=\OO_{2'}(G)\unlhd G$ and let $Nx$ be a real element of $G/N$ of order $k:=2^c\ge 2$. Then $Nx=Ny$ for some real element $y\in G$ by Lemma \ref{lem: real elements}(6).  We see that $y^{k}\in N$ and so $(y^k)^m=1$ for some odd integer $m\ge 1$. As $(k,m)=1$, $1=uk+vm$ for some integers $u,v$. We have $Nx=Ny=Ny^{uk+vm}=(Ny^{uk})(Ny^{vm})=Ny^{vm}$ as $y^k\in N$. Clearly $z:=y^{vm}$ is a nontrivial real element of $G$ whose order divides $k=2^c$ and $Nx=Ny=Nz$.
By Lemma \ref{lem:real 2-elements}, $z$ is a central involution of $G$ and thus $Nx$ is also a central involution of $G/N.$ 
\end{proof}

In the next theorem, we show that if a finite group satisfies the hypothesis of Theorem C, then it is solvable.
\begin{thm}\label{th:equal 2-parts-solvability}
Let $G$ be a finite group. Suppose that $|x^G|_2=2^a$ for all non-central real elements $x\in G$. Assume further that $G$ has a  Sylow $2$-subgroup $S$ with $\Real(S)\subseteq \Center(S)$. Then $G$ is solvable.
\end{thm}

\begin{proof}
Let $G$ be a minimal counterexample to the theorem and let $S$ be  a Sylow $2$-subgroup of $G$ with $\Real(S)\subseteq \Center(S)$.  Then $G$ is non-solvable and thus $G$ has no normal Sylow $2$-subgroup. By Lemma \ref{lem:even order real}, $G$ has a nontrivial real element $z$ of odd order. Clearly, $z$ is not central and thus $|z^G|$ is always even. Therefore, $|z^G|_2=2^a\ge 2.$

It follows from Lemma \ref{lem:central involutions} that every nontrivial real element of $G/\OO_{2'}(G)$ of $2$-power order lies in the center of $G/\OO_{2'}(G)$. In particular, all involutions of $G/\OO_{2'}(G)$ are in the center of $G/\OO_{2'}(G)$. Now we can apply results in \cite{Griess}. Since $\OO_{2'}(G/\OO_{2'}(G))=1$, by the main theorem in \cite{Griess}, the last term of the derived series of $G/\OO_{2'}(G)$, say $H/\OO_{2'}(G)$ is isomorphic to a direct product $L_1\times L_2\times\cdots\times L_n$, where each $L_i$ is isomorphic to either $\SL_2(q)$ with $q\ge 5$ odd or $2\cdot \Alt_7$, the perfect double cover of $\Alt_7$. 

For each $i$, every real element of $L_i$ is also a real element of $G/\OO_{2'}(G)$ as $L_i$ is a subgroup of $G/\OO_{2'}(G)$. Moreover, every nontrivial real element of $L_i$  of $2$-power order must lie in the center of $G/\OO_{2'}(G)$ and hence must be in $\Center(L_i).$ Thus to obtain a contradiction, we need to find a real element $x\in L_i$ of order $2^m\ge 4$. Notice that  $|\Center(L_i)|=2$ for all $i\ge 1$.

Assume first that $L_i\cong 2\cdot\Alt_7$ for some $i\ge 1.$ Using \cite{Atlas}, $L_i$ has a real element $x$ of order $4$. Assume next that $L_i\cong \SL_2(q)$ for some $q\ge 5$ odd. It is well known that the Sylow $2$-subgroup $T$ of $\SL_2(q)$ is a generalized quaternion group of oder $2^{k+1}$ for some $k\ge 2.$ (See, for example, Theorem 2.8.3 in \cite{Gorenstein}). Now $T$ is generated by two elements $\alpha$ and $\beta$ such that  $o(\beta)=2^k,o(\alpha)=4$, $\alpha^2=\beta^{2^{k-1}}$ and $\beta^\alpha=\beta^{-1}$. Thus $\beta$ is a real element of $\SL_2(q)$ of order $2^k\ge 4$. We can take $x=\beta$.
The proof is now complete.
\end{proof}

We now prove the $2$-nilpotence part of Theorem C.
\begin{thm}\label{th:equal 2-parts-2-nilpotent}
Let $G$ be a finite group. Suppose that $|x^G|_2=2^a$ for all non-central real elements $x\in G$. Assume further that $G$ has a  Sylow $2$-subgroup $S$ with $\Real(S)\subseteq \Center(S)$. Then $\OO^{2'}(G)$ is $2$-nilpotent.
\end{thm}

\begin{proof}
By Lemma \ref{lem:reduction}, we can assume that $G=\OO^{2'}(G)$. Let $\overline{G}=G/\OO_{2'}(G)$ and use the `bar' notation. Now $G$ is solvable by Theorem \ref{th:equal 2-parts-solvability}. Let $\overline{P}=\OO_{2}(\overline{G})$. 
Since $\overline{G}$ is solvable, it possesses a Hall $2'$-subgroup, say $\overline{H}$. It follows from Lemma \ref{lem:central involutions} that every real element of $\overline{G}$ whose order is a power of $2$ lies in the center of $\overline{G}$. This implies that $\overline{H}$ centralizes all real elements of order at most 4 of $\overline{P}$ and thus by \cite[Theorem B]{IN2}, $\overline{H}$ centralizes $\overline{P}$. By Hall-Higman 1.2.3, $\overline{H}\leq \Centralizer_{\overline{G}}(\overline{P})\leq \overline{P}$ which forces $\overline{H}=1$. This means that $\overline{G}=\overline{P}$ is a $2$-group and so $G$ is $2$-nilpotent as required.
\end{proof}

Finally, Theorem C follows by combining Theorems \ref{th:equal 2-parts-solvability} and \ref{th:equal 2-parts-2-nilpotent}.



\begin{thebibliography}{100}
\bibitem{BHM} Bertram, Edward A.; Herzog, Marcel; Mann, Avinoam. On a graph related to conjugacy classes of groups. \emph{Bull. London Math. Soc.} \textbf{22} (1990), no. 6, 569--575.

\bibitem{CM} Chillag, David; Mann, Avinoam. Nearly odd-order and nearly real finite groups. \emph{Comm. Algebra} \textbf{26} (1998), no. 7, 2041--2064. 


\bibitem{Atlas}  Conway, J. H.; Curtis, R. T.; Norton, S. P.; Parker, R. A.; Wilson, R. A. \emph{Atlas of finite groups}. Maximal subgroups and ordinary characters for simple groups. With computational assistance from J. G. Thackray. Oxford University Press, Eynsham, 1985.

\bibitem{DGPS} Dolfi, Silvio; Guralnick, Robert; Praeger, Cheryl E.; Spiga, Pablo. Coprime subdegrees for primitive permutation groups and completely reducible linear groups. \emph{Israel J. Math.} \textbf{195} (2013), no. 2, 745--772. 

\bibitem{DMN} Dolfi, Silvio; Malle, Gunter; Navarro, Gabriel. The finite groups with no real $p$-elements. \emph{Israel J. Math.} \textbf{192} (2012), no. 2, 831--840. 

\bibitem{DNT} Dolfi, Silvio; Navarro, Gabriel; Tiep, Pham Huu. Primes dividing the degrees of the real characters. \emph{Math. Z.} \textbf{259} (2008), no. 4, 755--774. 

\bibitem{DPS} Dolfi, Silvio; Pacifici, Emanuele; Sanus, Lucia. Finite groups with real conjugacy classes of prime size. \emph{Israel J. Math.} \textbf{175} (2010), 179--189. 

\bibitem{FA} Fisman, Elsa; Arad, Zvi. A proof of Szep's conjecture on nonsimplicity of certain finite groups. \emph{J. Algebra} \textbf{108} (1987), no. 2, 340--354. 

\bibitem{Gorenstein} Gorenstein, Daniel. \emph{Finite groups}. Second edition. Chelsea Publishing Co., New York, 1980.

\bibitem{Griess} Griess, Robert L., Jr. Finite groups whose involutions lie in the center. \emph{Quart. J. Math. Oxford Ser.} (2) \textbf{29} (1978), no. 115, 241--247. 

\bibitem{GNT} Guralnick, Robert M.; Navarro, Gabriel; Tiep, Pham Huu. Real class sizes and real character degrees. \emph{Math. Proc. Cambridge Philos. Soc.} \textbf{150} (2011), no. 1, 47--71.

\bibitem{Isaacs-1} Isaacs, I. Martin.  \emph{Finite group theory}. Graduate Studies in Mathematics, 92. American Mathematical Society, Providence, RI, 2008.

\bibitem{IN2} Isaacs, I. M.; Navarro, Gabriel. Normal $p$-complements and fixed elements.  \emph{Arch. Math}. (Basel) \textbf{95} (2010), no. 3, 207--211.

\bibitem{IN} Isaacs, I. M.; Navarro, Gabriel. Groups whose real irreducible characters have degrees coprime to $p$.  \emph{J. Algebra} \textbf{356} (2012), 195--206.

\bibitem{Tiep} Tiep, Pham Huu. Real ordinary characters and real Brauer characters.  \emph{Trans. Amer. Math. Soc.} \textbf{367} (2015), no. 2, 1273--1312.

\bibitem{Tong} Tong-Viet, H. P. Groups with some arithmetic conditions on real class sizes.  \emph{Acta Math. Hungar.} \textbf{140} (2013), no. 1-2, 105--116.
\end{thebibliography}
\end{document}